\documentclass{amsart}
\usepackage[T1]{fontenc}
\usepackage{enumerate}
\usepackage{amsthm,amsmath,amssymb}
\usepackage[all,cmtip]{xy}
\usepackage{mathrsfs}
\usepackage{amsfonts}
\usepackage[mathcal]{euscript}

\def \RN {{\mathbb {R}^N}}
\def \G {{\mathbb {G}}}
\def \R {{\mathbb {R}}}
\def \de {\partial}
\def \d {\mathrm{d}}
\def \LL {\mathcal{L}}
\def \RM {\check{H}}

\def \erre {{\mathbb {R}}}
\def \gi {{\mathbb {G}}}
\def \elle {{\mathcal {L}}}

\def \dive{\mathrm{div}}

\def\erren{{\erre^n}}

\def\inn{\ \mbox{ in }}

\newcommand{\ttende}{\longrightarrow}

\newtheorem{theorem}{Theorem}[section]
\newtheorem{proposition}[theorem]{Proposition}
\newtheorem{corollary}[theorem]{Corollary}
\newtheorem{lemma}[theorem]{Lemma}
\newtheorem{definition}[theorem]{Definition}
\newtheorem*{lemA}{Lemma A}
\newtheorem*{thmB}{Theorem B}
\newtheorem*{propC}{Proposition C}

\theoremstyle{remark}
\newtheorem{example}[theorem]{Example}

\theoremstyle{remark}
\newtheorem{remark}[theorem]{Remark}

\numberwithin{equation}{section}

\begin{document}
\title[Weighted ${L^p}$-Liouville Theorems]{Weighted ${L^p}$-Liouville Theorems\\ for Hypoelliptic Partial Differential\\ Operators on Lie Groups}

 \author{Andrea Bonfiglioli}
 \address{Dipartimento di Matematica,
         Universit\`{a} degli Studi di Bologna\\
         Piazza di Porta San Donato, 5 - 40126 Bologna, Italy}
 \email{andrea.bonfiglioli6@unibo.it}

\author{Alessia E. Kogoj}
 \address{Dipartimento di Matematica,
         Universit\`{a} degli Studi di Bologna\\
         Piazza di Porta San Donato, 5 - 40126 Bologna, Italy}
 \email{alessia.kogoj@unibo.it}

\begin{abstract}
 \noindent We prove \emph{weighted} $L^p$-Liouville theorems for a class of second order
 hypoelliptic partial differential operators $\LL$ on Lie groups $\G$ whose underlying manifold is
 $n$-dimensional space.
 We show that a natural weight is the right-invariant measure $\RM$ of $\G$.
 We also prove Liouville-type theorems for $C^2$ subsolutions in $L^p(\G,\RM)$.
 We provide examples of operators to which our results apply, jointly with an
 application to the uniqueness for the Cauchy problem for the evolution operator
 $\LL-\de_t$.
\end{abstract}
\keywords{Liouville Theorems; Degenerate-elliptic operators; Hypoelliptic operators on Lie groups}
\subjclass[2010]{Primary: 35B53; 35H10; 35R03; Secondary: 35H20; 35J70}


\maketitle

\section{Introduction and main results}\label{sec:intro}

 The aim of this paper is to obtain $L^p$-Liouville properties
 for hypoelliptic linear second order partial differential
 operators $\elle$ (with nonnegative cha\-racte\-ri\-stic form),
 which are left-invariant on a Lie group $\G$ on $n$-dimensional space $\R^n$.
 We shall obtain \emph{weighted} $L^p$-Liouville theorems, in that
 the right-invariant measure of the group $\G$ will play a crucial and natural r\^ole,
 as we will shortly explain.

  Precisely, we assume that $\LL$ has the following
  structure:
  $\LL$ is a linear second order PDO (with vanishing zero-order term) on $n$-dimensional space $\R^n$ whose
  quadratic form is positive semidefinite at every point of $\R^n$; more explicitly, to fix the notation, we require
  that $\LL$ has the coordinate form:
\begin{equation}\label{LLcoordinata}
 \elle = \sum_{i,j=1}^n a_{i,j}(x)\,\frac{\de^2}{\de{x_i}\de{x_j}}  + \sum _{j=1}^n  b_j(x) \frac{\de}{\de{x_j}},
\end{equation}
 with functions $a_{i,j},b_j\in C^\infty(\R^n,\R)$,
 and the matrix $A(x):=(a_{i,j}(x))$ is symmetric and positive semidefinite for every $x\in \R^n$.

 Our assumptions are the following three:
\begin{itemize}
 \item[(ND)] $\LL$ is \emph{non-totally degenerate} at every $x\in \R^n$, that is $A(x)\neq 0$
 for every $x\in \R^n$.

 \item[(HY)] $\elle$ is \emph{hypoelliptic} in every open subset of $\R^n$,
 that is, if $U\subseteq W\subseteq\RN$ are open sets, any $u\in \mathcal{D}'(W)$
 which is a weak solution to $\elle u=h$ in $\mathcal{D}'(U)$, with $h\in C^\infty(U,\R)$, is itself
 a smooth function on $U$.

 \item[(LI)] There exists a Lie group $\G=(\erren,\cdot)$ such that $\elle$ is \emph{left invariant} on $\G$.
\end{itemize}
\begin{remark}
 Actually, under hypothesis (LI), assumption (ND) is a very mild condition; indeed it is easy to check that, if $\LL$ is left invariant, then
 (ND) holds true if and only if
 there exists $x_0\in \R^n$ such that $A(x_0)\neq 0$ (which is equivalent to requiring that $\LL$ is not a merely first-order operator).

 We observe that a set of explicit necessary and sufficient conditions ensuring
 hy\-po\-the\-sis (LI) has been recently given by Biagi and the
 first-named author in \cite{BiagiBonfiglioli}.
\end{remark}

 We now fix a notation: in what follows we shall denote
 by $\RM$ a fixed right-invariant measure on the Lie group $\G$ in assumption (LI).
 Since any two right-invariant measures differ by a positive scalar multiple, we fix once
 and for all $\RM$ in the following (explicit way):
 given $x\in\R^n$ we set
 \begin{equation*}
    \rho_x:\R^n\to\R^n,\quad \rho_x(y):=y\cdot x
 \end{equation*}
 to denote the right translation by $x$; then it is easy to verify that the measure
\begin{equation}\label{traslaarhoRM}
  E\mapsto \RM(E):=\int_E \frac{1}{\det \big(\mathcal{J}_{\displaystyle\rho_x}(e)\big)}\,\d x\qquad
\end{equation}
 (defined on the Lebesgue measurable sets $E\subseteq\R^n$)
 is a right-invariant measure on the Lie group $\G$. Here and in the sequel we agree to denote
 by $\d x$ the Lebesgue integration on $\R^n$. The notation $\RM$ comes from the usual duality existing
 between left-invariant measures $\mu$ and right-invariant measures $\check{\mu}$:
 $$\mu\longmapsto \check{\mu}\quad \text{where $\check{\mu}(E)=\mu(\iota(E))$,} $$
 where $\iota$ is the group inversion on $\G$. Even if we will not use any Haar measure $H$
 of $\G$, we prefer to use the symbol $\RM$ to avoid any confusion with left invariance and, at the same time,
 in order to emphasize the r\^ole of right invariance in our Liouville results.

 Throughout, $L^p(\R^n,\RM)$ (for any $p\in [1,\infty]$) will
 denote the associated $L^p$-space on $\G \equiv \R^n$ with respect to the measure $\RM$.\medskip

 In the sequel, we say that a function $u\in C^2(\erren, \erre)$ is
 \begin{itemize}
   \item \emph{$\LL$-harmonic} on $\R^n$ if $\LL u=0$ on $\R^n$;
   \item  \emph{$\LL$-subharmonic} on $\R^n$ if it satisfies
 $\elle u \geq  0$ on $\erren$.
 \end{itemize}

 We are now ready to state the main results of this paper, the following weighted $L^p$-Liouville theorems.
\begin{theorem}[Weighted $L^p$-Liouville Theorem for the $\LL$-harmonic functions]\label{primo}
 Suppose that $\LL$ satisfies assumptions \emph{(ND)}, \emph{(HY)}, \emph{(LI)}.

 Let $u\in C^\infty (\erren, \erre)$ be an $\LL$-harmonic function.\medskip

Then $u\equiv 0$ if one of the following conditions is satisfied:
\begin{itemize}
 \item[{(i)}]  $u\in L^p(\erren,\RM)$ for some $p\in [1,\infty[$;
 \item[{(ii)}] $u\geq 0$ and $u^p\in L^1(\erren, \RM)$ for some $p\in\, ]0,1[.$
\end{itemize}
 $\RM$ denotes the right-invariant measure on $\G$ defined in \eqref{traslaarhoRM}.
\end{theorem}
  The classical form of Liouville's theorem for $\LL$-harmonic functions
 (i.e., under the assumption $\LL u=0$ and the  ``one-side'' bound $u\geq 0$ on the whole space) cannot be expected
 under our general hypotheses where operators with first-order terms as in \eqref{LLcoordinata} are allowed:
 for example, the classical Heat operator $$\LL=\sum_{j=1}^n (\de_{x_j})^2-\de_t\quad\text{in $\R^{n+1}=\R^n_x\times \R_t$}$$
 satisfies all the assumptions (ND), (HY) and (LI) (the latter w.r.t.\,the usual structure $\G=(\R^{n+1},+)$), but the function
 $\exp(x_1+\cdots+x_n+n\,t)$ is $\LL$-harmonic and nonnegative in space $\R^{n+1}$. 
 ``One-side'' Liouville-type theorems for some classes of homogeneous ope\-ra\-tors are proved in 
 \cite{kogoj_lanconelli1, kogoj_lanconelli2, kogoj_lanconelli3}.

 Our second main result, for $\LL$-subharmonic functions, is the following one:
\begin{theorem}[Weighted $L^p$-Liouville Theorem for the $\LL$-subharmonic functions]\label{secondo}
 Suppose that $\LL$ satisfies assumptions \emph{(ND)}, \emph{(HY)}, \emph{(LI)}.
 Let $u\in C^2(\erren, \erre)$ be an $\LL$-sub\-har\-mo\-nic function on $\R^n$.

 If $u\in L^p(\erren,\RM)$ for some $p\in [1,\infty[$, then  $u\le 0.$\medskip

 In particular, any nonnegative $\LL$-subharmonic function is identically zero,
 provided that $u\in L^p(\erren,\RM)$ for some $p\in [1,\infty[$.
\end{theorem}
 For the proofs of Theorems \ref{primo} and \ref{secondo} we closely follow the techniques
 recently introduced by Lanconelli and the second-named author in \cite[Th. 1.1, 1.2, 1.3]{kogoj_lanconelli_2014}, where
 unimodular Lie groups are considered (with $\RM$ equal to the Lebesgue measure):
 the ideas introduced in \cite{kogoj_lanconelli_2014} can be adapted to our (more general) framework,
 since they rely on a very versatile technique based on the use of convex functions of the global solution
 to $\LL u=0$, together with a general representation formula (of Poisson-Jensen type; see
 also \eqref{representation}). The novelty of our case is the use of the right-invariant measure $\RM$; this allows us
 to encompass new examples, of interest, as the following one.
\begin{example}\label{exe.colmo}
Let us consider in $\erre^{n+1}=\erre^n_x\times\erre_t$ the {\it
Kolmogorov-type operators}
\begin{equation}\label{kolmogorov}
 \elle = \dive (A\nabla) + \langle Bx,\nabla \rangle - \partial_t,\end{equation}
 where $A$ and $B$ are constant $n\times n$ real matrices, and $A$
 is  symmetric and positive semidefinite.
 Let us define the matrix
$$E(s):=\exp(-sB),\quad s\in\erre.$$
 Then the operator $\elle$ in \eqref{kolmogorov} satisfies assumption
 (LI) w.r.t.\,the Lie group $\mathbb{G}=(\erre^{n+1},\cdot)$
 with composition law
\begin{equation*}
 (x,t)\cdot (x',t')=(x'+E(t')x, t+t').
\end{equation*}
 Since $\det(E(t))=\exp(-t\,\textrm{trace}(B))$, according to formula \eqref{traslaarhoRM} the associated right-invariant
 measure $\RM$ is equal to
\begin{equation}\label{ipoipotesiMISU}
  \d\RM(x,t)= e^{t\,\textrm{trace}(B)}\,\d x\d t.
\end{equation}
 Moreover, if we assume that the matrix
\begin{equation}\label{ipoipotesi}
 \int_0^t E(s) \,A\, (E(s))^T\,\d s \quad\text{is positive definite for
 all $t>0$},
\end{equation}
 then $\elle$ is hypoelliptic (see e.g., \cite{lanconelli_polidoro_1994}; see also
  \cite[Sections 4.1.3, 4.3.4]{BLU}) so that hypothesis (HY) is satisfied as well. Condition \eqref{ipoipotesi} also
  encloses condition  (ND) (since \eqref{ipoipotesi} cannot hold if $A=0$).
  Hence, under condition \eqref{ipoipotesi}, the operator $\LL$
  satisfies all our assumptions and the weighted $L^p$-Liouville
  Theorems \ref{primo} and \ref{secondo} hold true w.r.t.\,the
  explicit measure $\RM$ in \eqref{ipoipotesiMISU}.\medskip

  For a class of operators (encompassing the above hypoelliptic operator $\LL$), we also prove
  a uniqueness result for the Cauchy problem (see Section \ref{sec:Cauchyuniq}); for simplicity we here state
 this result for the above operator $\LL$ (see Proposition \ref{pop:cauchyU}
 for the larger class of operators to which this uniqueness result applies):
\begin{corollary}\label{cor.tyko}
 Let us denote by $\Omega $ the half-space $\{(x,t)\in \R^{n+1}:t>0\}$.
  If $\LL$ is the operator \eqref{kolmogorov} and if
  the hypoellipticity condition \eqref{ipoipotesi} is satisfied,
  any classical solution $u\in C^\infty(\Omega )\cap C(\overline{\Omega })$
  to the Cauchy problem
  \begin{equation*}
    \left\{
      \begin{array}{ll}
        \LL u=0 &\hbox{in $\Omega $} \\
        u(x,t)=0&\hbox{for $t=0$}
      \end{array}
    \right.
  \end{equation*}
  is identically zero on $\Omega $ if it holds that
  $$\int_0^\infty \int_{\R^n} |u(x,t)|^p\,e^{t\,\textrm{\emph{trace}}(B)}\,\d x\,\d t<\infty, $$
  for some $p\in [1,\infty)$.
\end{corollary}
\end{example}
 Other examples, appearing in the literature, of operators
 satisfying conditions (ND), (HY) and (LI) are:\medskip
 \begin{enumerate}
  \item[(i)]
 the classical Kol\-mo\-go\-rov-Fokker-Planck operator
 $$\mathcal{K}=\sum_{j=1}^n (\de_{x_j})^2+\sum_{j=1}^{n} x_{j}\,\de_{x_{n+j}}-\de_t,$$
 in $\R^{2n+1}=\R^{2n}_x\times \R_t$ (it is of the form  \eqref{kolmogorov} and it satisfies \eqref{ipoipotesi});\medskip

 \item[(ii)]
 $\LL=L-\de_t$ in $\R^3=\R^2_x\times \R_t$, where $L=\tfrac{1}{2}\,(\de_{x_1})^2-(x_1+x_2)\,\de_{x_1}+x_1\,\de_{x_2}$
 ($L$ belongs to a class recently studied by Da Prato and Lunardi, \cite{DaPratoLunardiMATHZ});
 the associated right-invariant measure is
 $$e^{-t}\,\d t\,\d x_1\,\d x_2;$$

  \item[(iii)] the operators $\LL$ considered by
  Lanconelli and the first-named author in
  \cite{bonfiglioli_lanconelli_2012}, together with their evolution
  counterparts $\LL-\de_t$; since this class of PDOs furnishes a wide gallery of
  new examples for weighted $L^p$-Liouville theorems, we shall describe them in detail in Section \ref{sec:examples}.\bigskip
 \end{enumerate}
 Before giving the plan of the paper, we mention some related references from the existing literature:
\begin{itemize}
  \item[-]
 When hypothesis (LI) holds in the stronger form requiring that
 $\G$ is a homogeneous group w.r.t.\,a family of dilations (see \cite[Section 1.3]{BLU} for the relevant definition)
 and $\LL$ is a homogeneous operator, Theorem \ref{primo}
 follows from a general Liouville-type theorem by Geller \cite[Theorem 2]{Geller}.

  \item[-]
 Yet in presence of dilation-homogeneity (but not necessarily under the left-invariance condition (LI)),
 Luo extended Geller's theorem to homogeneous hypoelliptic operators (see \cite[Theorem 1]{Luo}).
 The theorems of Geller and of Luo cannot be applied to subharmonic functions (as in Theorem \ref{secondo} above).

  \item[-]
 For special classes of Lie groups $\G$ (namely, for stratified Lie groups), $L^1$-Liouville
 theorems on half-spaces have been proved
 by Uguzzoni \cite{Uguzzoni} and by the second-named author \cite{kogoj}.
 See also \cite{bonfiglioli_lanconelli_2001} (and \cite[Chapter 5, Section 5.8]{BLU})
 for Harnack-Liouville and asymptotic-Liouville theorems for stratified Lie groups.

 \item[-]
 The $L^\infty$-Liouville property does not hold, in general:
  see Priola and Zabczyk \cite{priola_zabczyk_2004} (see also
  \cite[Remark 8.1]{kogoj_lanconelli_2014}).

\end{itemize}
 The plan of the paper is as follows. Section \ref{sec:conseipo}
 recalls the techniques in \cite{kogoj_lanconelli_2014}, while
 in Section \ref{sec:prove} we prove Theorems \ref{primo} and \ref{secondo}.
 Finally, Section \ref{sec:examples} provides examples of operators to which
 our results apply, together with an application to the uniqueness of the Cauchy problem for a class
 of evolution operators.\bigskip

 \textbf{Acknowledgements.}\,\,The authors wish to thank Ermanno Lanconelli for
 having highlighted the problem and for useful conversations.
 The second-named author
 wishes to thank GNAMPA (Gruppo Nazionale per l'Analisi Matematica, la Probabilit\`a e le
 loro Applicazioni) of the Istituto Nazionale di Alta Matematica (INdAM)
 for supporting her investigation.
 \section{Background results and recalls}\label{sec:conseipo}
 Here and throughout the rest of the paper, we assume that $\LL$ is as in \eqref{LLcoordinata} and that
 the matrix $A(x)=(a_{i,j}(x))$ of the second order part of $\LL$ is symmetric and positive semidefinite
 for every $x\in \R^n$. This will be tacitly understood.
\begin{remark}\label{remproprietafurther}
 (a)\,\,Suppose that $\LL$ satisfies hypothesis (LI). Since
 (by the Campbell-Baker-Hausdorff Theorem; see e.g., \cite{BonfiglioliFulci})
 it is non-restrictive to assume that any Lie group is endowed with an analytic
 structure, then the coefficients of $\LL$
 can be supposed to be of class $C^\omega$.
 We shall assume the latter fact throughout.
 Moreover, by also using the Poincaré-Birkhoff-Witt Theorem,
 one can prove that assumption (LI) (together with the facts that the quadratic form
 of $\LL$ be positive semidefinite and be associated with a symmetric matrix) implies
 that $\LL$ is a sum of squares of vector fields plus a drift.
\medskip

 (b)\,\, We pass from \eqref{LLcoordinata}
  to the quasi-divergence form
\begin{equation}\label{LLcoordinataDIVE}
 \elle = \sum_{i=1}^n \frac{\de}{\de x_i} \Big(\sum_{j=1}^n a_{i,j}(x)\,\frac{\de}{\de{x_j}}\Big)
 + \sum _{j=1}^n \Big( b_j(x) -\sum_{i=1}^n \frac{\de a_{i,j}(x)}{\de x_i}\Big)\frac{\de}{\de{x_j}},
\end{equation}
 and we set
\begin{gather}\label{LLcoordinataDIVE2}
\begin{split}
 X_i &:=\sum_{j=1}^n a_{i,j}(x)\,\frac{\de}{\de{x_j}} \qquad (i=1,\ldots,n),\\
 X_0&:= \sum _{j=1}^n \Big( b_j(x) -\sum_{i=1}^n \frac{\de a_{i,j}(x)}{\de x_i}\Big)\frac{\de}{\de{x_j}}.
\end{split}
\end{gather}
 With this notation, \eqref{LLcoordinataDIVE} becomes
\begin{equation}\label{LLcoordinataDIVE3}
 \elle = \sum_{i=1}^n \frac{\de}{\de x_i}( X_i)+X_0.
\end{equation}
  If $\LL$ satisfies the hypoellipticity condition (HY),
  due to the results in \cite{OleinikRadkevic} (and the $C^\omega$ regularity in (a) above),
  then the vector fields $X_0,X_1,\ldots,X_n$ fulfil H\"ormander's maximal rank condition, \cite{Hormander}.

 Remarks (a) and (b) also motivate the fact that our examples of PDOs satisfying assumptions (HY) and (LI)
 (see Section \ref{sec:examples})
     will fall into the hypoellipticity class of the H\"ormander operators.
\end{remark}
 Then we fix a notation: if $A=(a_{i,j})$ is the second order matrix of $\LL$
 as in \eqref{LLcoordinata}, and if $u$ is of class $C^1$ on some open set, we set
\begin{equation}\label{gradientaA}
   \Psi_A(u)(x):=\sum_{i,j=1}^n a_{i,j}(x)\,\de_{x_i}u(x)\,\de_{x_j}u(x).
\end{equation}
 Notice that, since $A$ is positive semidefinite, one has $\Psi_A(u)=\langle A(x)\nabla u(x),\nabla u(x)\rangle\geq 0$.

 In \cite[Lemma 4.2]{kogoj_lanconelli_2014} it is proved the following result.
\begin{lemA}
 Let $\LL$ be as in \eqref{LLcoordinata} and let $\Psi_A$
 be as in \eqref{gradientaA}, where $A$ is the second order matrix of $\LL$.
 Suppose that the vector fields $X_0,X_1,\ldots,X_n$ in \eqref{LLcoordinataDIVE2} fulfil
 H\"ormander's maximal rank condition.\medskip

 Let $\Omega\subseteq\R^n$ be a connected open set and suppose that
 $u\in C^1(\Omega,\R)$.

 Then the following facts are equivalent:
\begin{enumerate}
  \item[\emph{(1)}] $u$ is constant on $\Omega$;
  \item[\emph{(2)}] $X_0 u,X_1 u,\ldots,X_nu$ all vanish on $\Omega$;
  \item[\emph{(3)}] $ \Psi_A(u)\equiv 0$ and $X_0\equiv 0$ on $\Omega$;
  \item[\emph{(4)}] $u$ is $\LL$-harmonic on $\Omega$ and $\Psi_A (u) \equiv 0$ on $\Omega$.
\end{enumerate}
\end{lemA}
 Due to its relevance in the sequel, we provide the proof of this lemma for the sake of completeness.

\begin{proof}
 Since (by hypothesis) $X_0,X_1,\ldots,X_n$ are bracket-generating vector fields, the
 equivalence of (1) and (2) follows from the well-known Connectivity Theorem of
  Carathéodory-Chow-Rashevsky (see e.g., \cite[Chapter 19]{BLU}).

 Next we recall that, given a symmetric positive semidefinite matrix $A$, then
 $$\langle A\xi,\xi\rangle =0 \quad \text{if and only if}\quad A\xi=0. $$
 As a consequence, $\Psi_A(u)(x)=0$ if and only if $\nabla u(x)$ is in the kernel of $A(x)$;
 but this latter condition (due to the very definition of $X_1,\ldots,X_n$) is equivalent to
 the fact that $X_1u(x)=\cdots=X_nu(x)=0$. Summing up,
\begin{equation}\label{annulareXconL}
 \Psi_A(u)\equiv 0\quad \text{if and only if}\quad X_1u,\ldots,X_nu\equiv 0.
\end{equation}
 Hence  (3) is equivalent to (2). Finally, the equivalence of (3) and (4)
 is a consequence of \eqref{LLcoordinataDIVE3}, taking into account \eqref{annulareXconL}.
\end{proof}
 The r\^ole of $\Psi_A(u)$ is clear from the following formula: if $u\in C^2(\Omega,\R)$ (for some open set $\Omega\subseteq\R^n$) and $F\in C^2(\R,\R)$ one has
 \begin{equation}\label{ruolPsi}
    \LL(F(u))=F'(u)\,\LL u+F''(u)\,\Psi_A(u).
 \end{equation}
 This formula has been exploited in \cite{kogoj_lanconelli_2014}, together with the 
 use of convex functions $F(u)$ of the global solution $u$ to $\LL u=0$, along with
 a representation formula of Poisson-Jensen type.

 The latter is recalled  in the next
 result, which is crucial for our purposes (see \cite[Theorem 2.3]{kogoj_lanconelli_2014} for the proof):
\begin{thmB}[Kogoj, Lanconelli, \cite{kogoj_lanconelli_2014}]
 Suppose that $\LL$ satisfies assumptions \emph{(ND)} and \emph{(HY)}.
 Then there exists a basis $\mathcal{B}$ for the Euclidean topology of $\RN$,
 whose elements are bounded open sets, with the following property:

 for every $\Omega\in \mathcal{B}$, and for every $x\in \overline{\Omega}$,
 there exist two Radon measures $\nu^\Omega_x$ on $\overline{\Omega}$
 and $\mu^\Omega_x$ on $\de{\Omega}$ such that,
 for any $v\in C^2(\overline{\Omega},\R)$, one has the representation formula
\begin{equation}\label{repre1}
    v(x)=\int_{\de{\Omega}} v(y)\,\d \mu_x^\Omega(y)-\int_{\overline{\Omega}} \LL v(y)\,\d \nu_x^\Omega(y),\quad \forall\,\,x\in \overline{\Omega}.
\end{equation}

 Moreover, if assumption \emph{(LI)} holds true,
 fixing a bounded open neighborhood $\Omega$ of $e$ (the neutral element of $\G$) as above,
 then we have
\begin{eqnarray}\label{representation}
 u(x)= \int_{\de\Omega} u(x \cdot y)\, \d\mu(y)  - \int_{\overline{\Omega}}
 (\elle u) (x \cdot y)\,\d\nu(y),
\end{eqnarray}
 for every $x\in\erren$ and every $u \in C^2(\erren, \erre)$. Here
 we have set, for brevity,
\begin{equation}\label{notazfix}
 \nu:=\nu^\Omega_e,\quad \mu:=\mu^\Omega_e.
\end{equation}
\end{thmB}
 In view of the central use of representation formula \eqref{representation}, we fix some notation.
\begin{definition}\label{defi.mean}
  For any $u\in C(\R^n,\R)$ and any $x\in \R^n$, we set
\begin{gather}\label{operatori}
\begin{split}
 M(u)(x) &:= \int_{\de\Omega} u(x \cdot y)\, \d \mu(y),\\
 N(u)(x) &:= \int_{\overline{\Omega}} u(x \cdot y)\, \d \nu(y).
\end{split}
\end{gather}
\end{definition}
 Hence \eqref{representation} can be written as follows
\begin{eqnarray}\label{compareprese}
 u(x)= M(u)(x) - N(\elle u)(x) \qquad \forall \ x\in\erren,\quad \forall\,\, u\in C^2(\R^n,\R).
\end{eqnarray}
 Distinguished properties of the operators $M,N$ are proved in \cite[Lemma 3.2]{kogoj_lanconelli_2014},
 which we here recall:
\begin{propC}[Kogoj, Lanconelli, \cite{kogoj_lanconelli_2014}]
 Let $u\in C(\erren, \erre)$ and let $M$ and $N$ be the operators in \eqref{operatori}.
\begin{itemize}
\item[{(i)}] If $u\geq 0$ then $M(u),N(u) \geq 0$;

\item[{(ii)}] $M(u),N(u)\in C(\erren,\erre)$.

\item[{(iii)}] If $N(u)\equiv 0$ (or $M(u)\equiv 0$) and  $u\geq 0$, then $u\equiv 0$.
\end{itemize}
\end{propC}
\section{Proof of the weighted $L^p$-Liouville theorems}\label{sec:prove}
 For the rest of the paper, we assume that $\LL$ satisfies assumption (ND), (HY) and (LI).
 A main tool in the proof of our $L^p$-Liouville theorems is the following Lemma \ref{elleunoelleuno}.
 It shows the r\^ole of the right-invariant measure $\RM$ with respect to the operator $M$.
 Lemma \ref{elleunoelleuno} and Corollary \ref{corollarytredue} are the versions, respectively,
 of \cite[Lemma 3.1]{kogoj_lanconelli_2014} and of \cite[Proposition 4.3]{kogoj_lanconelli_2014},
 where we drop the assumptions that $\G$ be unimodular and that $\RM$ be the Lebesgue measure.
\begin{lemma}\label{elleunoelleuno}
 Let $u\in C(\erren, \erre)$ be such that $u\in L^1(\erren, \RM)$, where $\RM$ is the right-invariant measure
 on $\G=(\R^n,\cdot)$ introduced in \eqref{traslaarhoRM}.

 Then $M(u)\in L^1(\erren, \RM)$ and
\begin{eqnarray}\label{uguagliosaMu}
 \int_\erren M(u)(x)\,\d\RM(x) = \int_\erren u(x)\,\d\RM(x).
\end{eqnarray}
\end{lemma}
\begin{proof}
  It is a consequence of Fubini Theorem. We skip the proof of the fact that $M(u)\in L^1(\erren, \RM)$, since it follows
 by a similar argument as the following one. We have:
\begin{align*}
 \int_\erren M(u)(x)\,\d \RM (x) &= \int_\erren \left( \int_{\de\Omega} u(x\cdot y)\,\d \mu(y)\right)\,\d \RM (x)\\
&= \int_{\de\Omega} \left( \int_\erren u(x\cdot y)\,\d \RM (x) \right)\,\d\mu(y)\\
 \mbox{ ($\RM $ is right invariant on $\gi$) } \quad &=
 \int_{\de\Omega} \left( \int_\erren u(x)\,\d \RM (x) \right)  \d\mu(y)\\
& =\left(\int_\erren u(x)\,\d \RM (x) \right) \left(\int_{\de\Omega} \d\mu(y)\right)\\
&= \int_\erren u(x)\,\d \RM (x) .
\end{align*}
 In the last equality we have used identity $\mu (\de\Omega)=1$, coming from
 \eqref{representation} with $u\equiv 1$.
\end{proof}
\begin{corollary} \label{corollarytredue}
 Let $u \in C^2(\erren, \erre)$ be an $\LL$-subharmonic function.

 If  $u\in L^1(\erren, \RM)$, then $u$ is actually $\LL$-harmonic on $\R^n$.
\end{corollary}
\begin{proof}
 From \eqref{compareprese} we have
 $N(\LL u)= M(u)-u$ on $\R^n$.
 By Lemma  \ref{elleunoelleuno}, $u\in L^1(\erren, \RM)$ implies $M(u)\in L^1(\erren, \RM)$, whence
  $N(\LL u)\in L^1(\erren, \RM)$ too. From \eqref{uguagliosaMu} we also get
\begin{eqnarray*}
  \int_\erren N (\elle u)\, \d\RM=\int_\erren M(u)\, \d\RM-\int_\erren u\, \d\RM=0.
\end{eqnarray*}
 On the other hand, since $\elle u \geq 0$, we have
 $N(\elle u) \geq 0$ in $\erren$ (see
 Proposition C-(i)). Therefore
 $N(\elle u)= 0$ $\RM$-almost-everywhere in $\erren$.
 Since  $\RM$ is equal to a (smooth) positive  density times the Lebesgue measure
 on $\R^n$ (see \eqref{traslaarhoRM}), we infer that
\begin{equation}\label{quasidappLeb}
 \text{$N(\elle u)= 0$ \quad Lebesgue-almost-everywhere in $\erren$.}
\end{equation}
 From $u\in C^2$, we get $\elle u\in C$ so that, by Proposition C-(ii),
 $N(\elle u)$ is continuous.
 As a consequence of \eqref{quasidappLeb} it follows $N(\elle u)\equiv 0$.
 Finally, the $\LL$-subharmonicity of $u$ and an application of Proposition C-{(iii)}
 shows that $u$ is $\LL$-harmonic in $\erren$.
\end{proof}

 Now, we are in the position to prove Theorems \ref{primo} and \ref{secondo} proceeding
 along the lines of \cite{kogoj_lanconelli_2014}.
 First we need a result from Lie-group theory:
 this comes from the characterization of compact groups in terms of the finiteness of the Haar measure
 (see e.g., \cite[Proposition 1.4.5]{DeitmarEchterhoff}). We give the (very short) details for completeness.
\begin{lemma}\label{summab.HHHH}
 The only constant function belonging to $L^1(\R^n,\RM)$ is the null function.
\end{lemma}
\begin{proof}
 We argue by contradiction: we assume the existence of a non-vanishing constant function in
 $L^1(\R^n,\RM)$, which is equivalent to requiring that $\RM(\R^n)<\infty$. If this happens, we can find
 a compact neighborhood $U$ of the neutral element $e$ of $\G$, and at most a finite family
 of mutually disjoint sets
 $$U\cdot x_1,\ldots, U\cdot x_k,\quad \text{with $k$  maximal}.$$
 Here we have used the right invariance of $\RM$, ensuring that
 $\RM(U\cdot x_i)=\RM(U)>0$, for any $i=1,\ldots,k$.
 We set $K:=\bigcup_{i=1}^k U\cdot x_i$, which is clearly a compact set in $\R^n$.

 From the maximality of $k$, it is simple to recognize that, for any $x\in \R^n$, one has $K\cap (K\cdot x)\neq \emptyset$.
 This shows that $\R^n=K^{-1}\cdot K$, which is absurd since the latter is a compact set.
 Hence $\RM(\R^n)=\infty$.
\end{proof}
 We are ready to give the proofs of our main results.\bigskip

\begin{proof}\emph{(of Theorem \ref{primo}.)}
Let $u$ be a (smooth) solution to $\elle u = 0$  in $\erren$.\medskip

 (i)\,\,Assume  $u\in L^p(\erren, \RM)$ (for some $1\le p <\infty$) and consider
 $v:= F(u)$, where $$F:\erre\longrightarrow\erre,\quad  F(t)=(\sqrt{1+t^2} - 1)^p .$$
 It is easy to check that
\begin{itemize}
  \item[-] $F\in C^2(\R,\R)$;

  \item[-] $0\leq F(t)\leq |t|^p$ for every $t\in\R$;

  \item[-] $F''(t)>0$ for every $t\neq 0$.
\end{itemize}
 Then  $v \in C^2(\erre,\erre)$, $v\in L^1(\erren, \RM)$ (since $u\in L^p(\erren, \RM)$ and $|F(u)|\leq |u|^p$) and
\begin{equation*}
 \begin{split}
 \elle v  & \stackrel{\eqref{ruolPsi}}{=} F' (u)\,\elle u + F''(u)\,\Psi_A(u)=F''(u)\,\Psi_A(u)\geq 0.
\end{split}
\end{equation*}
 Therefore, by Corollary \ref{corollarytredue}, $\elle v =0$ so that, since
 $F''(u) >0$ if $u\neq 0$,
\begin{equation}\label{eqduecinque}
 \Psi_A(u)=0\inn \Omega_0:=\{x\in\erren\ |\ u(x)\neq 0\}.\end{equation}
 If $\Omega_0 = \emptyset $ we are done, since we aim to prove that $u\equiv 0$.
 Assume, by contradiction, that $\Omega_0 \neq \emptyset$.  Keeping in mind that
 $\elle u = 0$ in $\erren$ by hypothesis and that $\Psi_A(u)=0$ on $\Omega_0$ by construction,
 from Lemma A-(4) we get
 that $u$ is constant on every non-empty connected component $O$ of  $\Omega_0$.

 If $\partial O\neq \emptyset$, since $\de O\subseteq \de \Omega_0$ (and clearly $u=0$ on $\de\Omega_0$),
 then $u\equiv 0$ in $O$, in contradiction with the very definition of  $\Omega_0.$
 Thus,  $\partial O= \emptyset$, i.e.,  $\Omega_0=\erren$ and $u$ is constant on $\R^n$.
 Now, the assumption $u\in L^p(\erren, \RM)$ jointly with Lemma \ref{summab.HHHH}, shows that $u\equiv 0$,
 in contradiction with $\Omega_0\neq \emptyset$. This ends the proof of Theorem \ref{primo} under assumption (i).\medskip

 (ii)\,\,Assume  $u\geq 0$ and  $u^p\in L^1(\erren, \RM)$ for some $p\in\, ]0,1[.$ Define $v:= F(u)$, with
 $$F: [0, \infty[ \longrightarrow\erre,\quad  F(t)=(1 + t)^p - 1.$$
 $F$ has the following properties:
\begin{itemize}
  \item[-] $F\in C^\infty([0,\infty),\R)$;

  \item[-] $0\leq F(t)\leq t^p$ for every $t\geq 0$;

  \item[-] $F''(t)<0$ for every $t\geq 0$.
\end{itemize}
 Therefore $v \in C^\infty(\erren,\erre)$, $v\in L^1(\erren,\RM)$ and
\begin{equation*}
\begin{split}
 \elle v  & \stackrel{\eqref{ruolPsi}}{=}
 F' (u)\,\elle u + F''(u)\,\Psi_A(u)=  F''(u)\,\Psi_A(u)\le 0 .
\end{split}
\end{equation*}
 Thus, by Corollary \ref{corollarytredue} applied to $-v$, we infer that $\elle v=0$  in $\erren$ so that
 the above identity yields $0=F''(u)\,\Psi_A(u)$. As a consequence,
 since $F''(u)<0$ (recall that $u\geq 0$ by assumption), we get
\begin{equation*}
 \Psi_A(u)\equiv 0\inn \erren.
\end{equation*}
 Since $\LL u=0$ in $\R^n$ by hypothesis,
 a direct application of Lemma A-(4) proves that $u$ is constant
 in $\R^n$. Since $u^p$ belongs to $L^1(\erren, \RM)$,  we are entitled to apply
 Lemma \ref{summab.HHHH} and infer that $u\equiv 0$ in $\R^n$, and this ends the proof.
\end{proof}

 We end the section with the proof of our weighted $L^p$-Liouville Theorem
 for the $\LL$-subharmonic functions.\medskip

\begin{proof}\emph{(of Theorem \ref{secondo}.)}
 Let $u\in C^2(\erren,\erre)$ be $\elle$-subarmonic and let it belong to $L^p(\erren, \RM)$
 (for some $p\in [1,\infty)$). We aim to prove that
 $$\Omega_+:=\{ x\in\erren\ | \  u(x)>0\}=\emptyset.$$
 We argue
 by contradiction and assume that $\Omega_+\neq\emptyset$.
 Let us consider the function
 $$F:\erre\ttende\erre,\quad
 F(t) := \left\{
           \begin{array}{ll}
             0 & \hbox{if $t\leq 0$,} \\
             \big(\sqrt[4]{{1 + t^4}} - 1\big)^{p} & \hbox{if $t>0$.}
           \end{array}
         \right.$$
 It is easy to recognize that:
\begin{itemize}
\item[(i)] $F\in C^2(\R,\R)$, $F$ is increasing and convex;
\item[(ii)] $F'>0$ and $F''>0$ in $]0,\infty[$;
\item[(iii)] $0\le F(t)\le t^p$ for every $t\geq 0$.
 \end{itemize}
 We define $v:=F(u)$ on $\R^n$. Then $v\in C^2(\erre,\erre)$ and, by property (iii) above,
 $v\in L^1(\erren, \RM)$.
 Moreover, by identity \eqref{ruolPsi},
 $$\elle v= F'(u)\, \elle u + F''(u)\,\Psi_A(u)\geq 0,$$
 since $\elle u \geq 0$ and $F',F''\geq 0$  by (i).
 Summing up, $v$ is $\LL$-subharmonic in space and in $L^1(\erren, \RM)$:
 Corollary \ref{corollarytredue} then implies that $\elle v\equiv 0$, whence
\begin{equation*}
 F'(u) \,\elle u + F''(u)\,\Psi_A(u)= 0\quad \text{in $\R^n$}.
\end{equation*}
 Taking into account property (ii) of $F$, we obtain
\begin{equation*}
 \elle u =0\quad\text{and}\quad \Psi_A(u)= 0 \quad \text{in $\Omega_+$}.
\end{equation*}
  We are therefore entitled to apply Lemma A-(4) on every connected component
 $O$ of $\Omega_+$, and derive that $u$ is constant on $O$.

 If $\partial O\neq \emptyset$, since $\de O\subseteq \de \Omega_+$ (and clearly $u=0$ on $\de\Omega_+$),
 then $u\equiv 0$ in $O$, in contradiction with the definition of  $\Omega_+$.
 Thus,  $\partial O= \emptyset$, so that $\Omega_+=\erren$ and $u$ is constant on $\R^n$.
 As in the proof of Theorem \ref{primo}, by invoking Lemma \ref{summab.HHHH} we get that $u\equiv 0$,
 in contradiction with $\Omega_+\neq \emptyset$.\medskip
\end{proof}

\section{Examples and an application}\label{sec:examples}
 We now give new examples of PDOs to which the $L^p$-Liouville theorems
 apply.
\subsection{Matrix-exponential groups}\label{sec:MEXPgroups}
 We denote the points of $\R^{1+n}$ by $(t,x)$, with $t\in \R$ and $x\in \R^n$.
 Let $B$ be a real square matrix of order $n$; following
 \cite[Section 2]{bonfiglioli_lanconelli_2012}, we say that
 the \emph{ma\-trix-ex\-po\-nen\-tial group} $\G(B)$ related to $B$ is $(\R^{1+n},\cdot)$
 endowed with the product
\begin{equation*}
 (t,x)\cdot (t',x')= \big(t+t',x+\exp(t\,B)\,x'\big),\qquad
 t,t'\in\R,\,\,x,x'\in\R^n.
\end{equation*}
 A basis for the Lie algebra of $\G(B)$, say $\mathfrak{g}(B)$,
 is $\{\de_t,X_1,\ldots,X_n\}$,
 where
 $$\textstyle X_j:= \sum_{k=1}^n  a_{k,j}(t)\,\partial_{x_k}\quad (j= 1,\dots, n),$$
 where $a_{k,j}(t)$ is the entry of position $(k,j)$ of the matrix
 $\exp(t\,B)$.  The neutral element of $\G(B)$ is $(0,0)$.
 The right-invariant measure $\RM$ in \eqref{traslaarhoRM} is equal
 to the Lebesgue measure in $\R^{1+n}$, since (in block form) we
 have
 $$\mathcal{J}_{\displaystyle \rho_{(t',x')}}((0,0))=
 \left(
   \begin{array}{cc}
     1 & 0 \\
     Bx' & \mathbb{I}_n \\
   \end{array}
 \right)\qquad (\text{$\mathbb{I}_n$ is the identity matrix of order
 $n$}).$$
 On the other hand, since any Haar measure on $\G(B)$ is a (positive) scalar multiple of
 $$\frac{1}{\det\mathcal{J}_{\displaystyle \tau_{(t,x)}}((0,0))}\,\d t\d x $$
 (here $\tau_{(t,x)}$ denotes the left translation by $(t,x)$),
 and since
 \begin{equation}\label{haarGB}
    \mathcal{J}_{\displaystyle \tau_{(t,x)}}((0,0))=
     \left(
   \begin{array}{cc}
     1 & 0 \\
     0 & \exp(tB) \\
   \end{array}
 \right),
 \end{equation}
 then $\G(B)$ is unimodular (with $\d t\d x$ as left/right-invariant measure)
 if and only if $\mathrm{trace}(B)=0$.
 Thus, our results here are contained in \cite{kogoj_lanconelli_2014} \emph{only when
 $\mathrm{trace}(B)=0$}.\medskip

 More precisely, the $L^p$-Liouville Theorems \ref{primo} and
 \ref{secondo} hold true \emph{for any matrix $B$} (with $L^p$ standing for the usual $L^p(\R^{1+n})$ space),
 and for any second order operator $\LL$ which is a polynomial of degree $2$ in
  $\de_t,X_1,\ldots,X_n$ (hence it is left invariant), provided that $\LL$ also fulfils
  our structure hypotheses (ND) and (HY). For example, $\LL$ may be of the form
  $$X_1^2+\cdots+X_n^2+(\de_t)^2 $$
  (a sum of square of H\"ormander vector fields), or of the forms
  $$X_1^2+\cdots+X_n^2-\de_t,\qquad X_1^2+\cdots+X_n^2+\de_t $$
   (evolution H\"ormander operators, with drift terms $\pm\de_t$).
   These operators are non-de\-ge\-ne\-ra\-te (recall that the vector fields $X_1,\ldots,X_n$
   are associated with the columns of the non-null matrix
   $\exp(tB)$), and they are hypoelliptic, due to H\"ormander
   hypoellipticity condition (since $\{\de_t,X_1,\ldots,X_n\}$ is a basis of
   $\mathfrak{g}(B)$).

 More degenerate operators are allowed, as in the next example.
\begin{example}\label{exa.companion}
  Following \cite[Section 3]{bonfiglioli_lanconelli_2012}, if $B$ takes on the special
  ``companion'' form
  $$
 B=\begin{bmatrix}
 0     &  0   & \cdots & 0    & -a_0 \\
 1     &  0   & \cdots & 0    & -a_1 \\
 0     &  1   & \cdots & 0    & -a_2 \\
 \vdots&\vdots& \ddots &\vdots& \vdots \\
 0     &  0   & \cdots & 1    & -a_{n-1}
 \end{bmatrix}
 $$
 (for some assigned real numbers $a_0,a_1,\ldots,a_{n-1}$), then
 $$\exp(tB)=
\begin{bmatrix}
 u_1(t)    &u'_1(t)     &\cdots & u_{1}^{(n-1)}(t) \\
 \vdots & \vdots  &\vdots &\vdots \\
 u_n(t) & u'_n(t) &\cdots & u_n^{(n-1)}(t)
\end{bmatrix},
$$
 where $\{u_1(t),\ldots,u_n(t)\}$ is the fundamental system of solutions of the $n$-th order
  constant-coefficient ODE
\begin{equation*}
 u^{(n)}(t)+a_{n-1}\,u^{(n-1)}(t)+\cdots +a_{1}\,u'(t)+a_0\,u(t)=0.
\end{equation*}
 Following our previous notation for the basis $\{\de_t,X_1,\ldots,X_n\}$ of $\mathfrak{g}(B)$, we have
 $$X_j=u^{(j-1)}_1(t)\,\de_{x_1}+\cdots+u^{(j-1)}_n(t)\,\de_{x_n}\quad (j=1,\ldots,n). $$
 This shows that
 it is sufficient to consider the two vector fields
 $$\de_t\quad\text{and}\quad X_1=u_1(t)\,\de_{x_1}+\cdots+u_n(t)\,\de_{x_n}$$
 to Lie-generate the whole of $\mathfrak{g}(B)$. Therefore, the five operators
  $$(\de_t)^2+(X_1)^2,\quad (\de_t)^2\pm X_1,\quad (X_1)^2\pm\de_t $$
  are H\"ormander operators (hence hypoelliptic) to which our $L^p$-Liouville results apply
  (with measure $\RM$ equal to the Lebesgue measure in $\R^{1+n}$).
  When $a_{n-1}\neq 0$, the associated matrix-exponential group $\G(B)$ is \emph{not} unimodular, so that these operators are not comprised in \cite{kogoj_lanconelli_2014}.
\end{example}
\subsection{The inverse group of $\G(B)$}\label{sec:MEXPgroupsINV}
 We have the following examples:\medskip

 (a)\,\,Let $\G(B)$ be the group constructed in Section \ref{sec:MEXPgroups}.
 Following \cite[Section 4]{bonfiglioli_lanconelli_2012}, we can interchange
 right and left multiplications of $\G(B)$, obtaining the group
 $\widehat{\G}(B):=(\R^{1+n},\widehat{\cdot}\,)$ (referred to as the \emph{inverse group of $(\G(B),\cdot)$}),
  where
\begin{equation}\label{p.compo.BBIN}
 (t,x)\,\widehat{\cdot}\, (t',x')= \big(t+t',x'+\exp(t'\,B)\,x\big),\qquad
 t,t'\in\R,\,\,x,x'\in\R^n.
\end{equation}
 Hence, a basis for the Lie algebra of $\widehat{\G}(B)$ is
 $\{T,\de_{x_1},\ldots, \de_{x_n}\}$, where
 $$T:=\de_t+\sum_{i,j=1}^n b_{i,j}\,x_j\,\de_{x_i}.$$
 The neutral element of $\widehat{\G}(B)$ is $(0,0)$ and
 the associated right-invariant measure $\RM$ in \eqref{traslaarhoRM} is equal
 to
\begin{equation}\label{p.compo.BBRMM}
 \d\RM(t,x)=e^{-t\,\textrm{trace}(B)}\,\d t\d x,
\end{equation}
 and it is easy to recognize that
 $\widehat{\G}(B)$ is unimodular (with $\d t\d x$ as left/right-invariant measure)
 if and only if $\mathrm{trace}(B)=0$.

 As a consequence, the $L^p$-Liouville Theorems \ref{primo} and
 \ref{secondo} hold true for any matrix $B$, with $L^p$ standing for $L^p(\R^{1+n},\RM)$
 with $\RM$ as in \eqref{p.compo.BBRMM}, and
 for any second order operator $\LL$ which is a polynomial of degree $2$ in
 $T,\de_{x_1},\ldots,\de_{x_n}$, provided that $\LL$ also fulfils
 our structure hypotheses (ND) and (HY). For example, if we use the compact notations
 $\Delta_x:=\sum_{j=1}^n (\de_{x_j})^2$ and $T=\de_t+\langle Bx,\nabla_x \rangle$, $\LL$ may be of the form
  $$\LL_1=\Delta_x+\Big(\de_t+\langle Bx,\nabla_x \rangle\Big)^2 $$
  (a sum of square of H\"ormander vector fields), or of the form (replacing $B$ with $-B$)
  $$\LL_2=\Delta_x+\langle Bx,\nabla_x \rangle -\de_t$$
   (an evolution H\"ormander operator, with drift term $\langle Bx,\nabla_x \rangle -\de_t$), which is
   a left-invariant evolution PDO of
   Kolmogorov-Fokker-Planck type.\medskip

 (b)\,\,Many other examples inspired by the previous case are available of more de\-ge\-ne\-ra\-te operators to which our results apply:
 for example,
 when
 $$ B=\left(
        \begin{array}{cc}
          0 & 0 \\
          1 & 0 \\
        \end{array}
      \right),$$
 we have $T=\de_t+x_1\,\de_{x_2}$ and it is sufficient to consider $\de_{x_1}$
 to obtain the H\"ormander system $\{T,\de_{x_1}\}$ in $\R^3=\R_t\times \R^2_x$. As a consequence, our
 Theorems \ref{primo} and \ref{secondo} (with $\RM$ equal to the Lebesgue measure on $\R^3$) apply to the two H\"ormander operators
 $$\LL_1=(\de_{x_1})^2+(\de_t+x_1\,\de_{x_2})^2,\quad \LL_2=(\de_{x_1})^2-\de_t-x_1\,\de_{x_2}.$$
 The associated group law \eqref{p.compo.BBIN} is
\begin{equation*}
 (t,x_1,x_2)\,\widehat{\cdot}\, (t',x_1',x_2')= \Big(t+t',x_1+x'_1,x_2+x'_2+x_1t'\big),
\end{equation*}
 which defines the so-called polarized Heisenberg group.
 We remark that $\LL_2$ is a de\-ge\-ne\-ra\-te ultraparabolic operator of Kolmogorov-Fokker-Planck type.

 More generally, one can consider
 a matrix of the form
 $$ B=\left(
        \begin{array}{cc}
          0 & 0 \\
          \mathbb{I}_n & 0 \\
        \end{array}
      \right) \quad \text{($\mathbb{I}_n$ is the $n\times n$ identity matrix)},$$
 and the associated $\widehat{\mathbb{G}}(B)$ group. In this case the measure
 $\RM$ is the Lebesgue measure on $\R^{2n+1}$ and a meaningful operator to which
 our results apply is
 $$\mathcal{K}=\sum_{j=1}^n (\de_{x_j})^2+\sum_{j=1}^{n} x_{j}\,\de_{x_{n+j}}-\de_t,$$
 the classical Kolmogorov-Fokker-Planck operator (see example (i) in the Introduction). \medskip

 (c)\,\,Yet another example is given by the matrix
 $$ B=\left(
        \begin{array}{cc}
          1 & 1 \\
          -1 & 0 \\
        \end{array}
      \right),$$
 so that $T=\de_t+(x_1+x_2)\de_{x_1}-x_1\de_{x_2}$ and the associated right-invariant measure is
 $$\d\RM(t,x_1,x_2)= e^{-t}\,\d t\d x_1\d x_2.$$
 An operator fulfilling our hypotheses (ND), (HY), (LI) is therefore
 $$\LL=\Big(\tfrac{1}{\sqrt2}\,\de_{x_1}\Big)^2-T=\tfrac{1}{2}(\de_{x_1})^2
 -(x_1+x_2)\de_{x_1}+x_1\de_{x_2}-\de_t,$$
 considered in example (ii) in the Introduction.\medskip

 (d)\,\,In general, if the operator $\LL$ satisfies our structure conditions
  (ND), (HY), (LI) (the latter w.r.t.\,the group $\G=(\R^n,\cdot)$), we can
  add an extra variable $t\in \R$ thus obtaining a
  new evolution operator
 $$\mathcal{H}:=\LL-\de_t\quad \text{ on $\R^{n+1}=\R^{n}_x\times \R_t$,}$$
  to which Theorems
 \ref{primo} and \ref{secondo} can be applied: it suffices to consider the Lie group obtained as a direct product
 of ${\G}$ with the group $(\R_t,+)$, and by taking into account the right-invariant product measure $$\d \RM(x)\,\d t.$$
\subsection{An application to the uniqueness of the Cauchy problem}\label{sec:Cauchyuniq}
 Suppose that the operator $\LL$ in $\R^{1+n}$ (whose points are denoted by
 $(t,x)$, with $t\in \R$ and $x\in \R^n$) satisfies our structure assumptions
 (ND), (HY), (LI); assume furthermore that $\LL$ has the following ``Heat-type'' form:
\begin{equation}\label{formaLLLLCauchy}
 \elle = L-\de_{t},\quad \text{where}\quad
 L=\sum_{i,j=1}^{n} a_{i,j}(x)\,\frac{\de^2}{\de{x_i}\de{x_j}}  + \sum _{j=1}^{n}  b_j(x)\frac{\de}{\de{x_j}}.
\end{equation}
 Then we prove the following uniqueness result:
\begin{proposition}\label{pop:cauchyU}
 Let us denote by $\Omega $ the half-space $\{(t,x)\in \R^{1+n}:t>0\}$.
 Under the above assumptions and notation on $\LL$, any classical solution $u\in C^\infty(\Omega )\cap C(\overline{\Omega })$
  to the Cauchy problem
  \begin{equation}\label{CAUUUeq}
    \left\{
      \begin{array}{ll}
        \LL u=0 &\hbox{in $\Omega $} \\
        u(t,x)=0&\hbox{for $t=0$}
      \end{array}
    \right.
  \end{equation}
  is identically zero on $\Omega $ if it holds that $u\in L^p(\Omega,\RM)$ for some $p\in [1,\infty)$.
 As usual, $\RM$ denotes the right-invariant measure \eqref{traslaarhoRM} on the group $\G$ for which
  hypothesis \emph{(LI)} holds.
\end{proposition}
\begin{proof}
 Let us denote by $\overline{u}$ the trivial prolongation of $u$ on $\R^{1+n}$ obtained
 by setting $\overline{u}$ to be $0$ when $t<0$. Clearly, $u\in C(\R^{1+n},\R)\cap L^p(\R^{1+n},\RM)$,
 as $u\in L^p(\Omega,\RM)$. We claim
  \begin{equation}\label{CAUUUeq2}
  \overline{u}\in C^\infty(\R^{1+n},\R) \quad \text{and}\quad \LL \overline{u}=0 \,\,\text{on $\R^{1+n}$.}
  \end{equation}
 Once we have proved this, an application of Theorem \ref{primo} to $\overline{u}$ will prove that
 $\overline{u}\equiv 0$ on $\R^{1+n}$, i.e., $u=0$ on $\Omega$.

 We are then left to prove the claimed  \eqref{CAUUUeq2}. Since $\LL$ is hypoelliptic by assumption (HY),
 \eqref{CAUUUeq2} will follow if we show that $\LL \overline{u}=0$ on $\R^{1+n}$ in the weak sense of distributions.
 To this aim, let $\varphi\in C_0^\infty(\R^{1+n})$. We have the following computation:
\begin{align*}
    &\int_{\R^{1+n}} \overline{u}\,\LL^*\varphi=\int_{\R^{1+n}} \overline{u}\,(L^*\varphi+\de_{t}\varphi)=
  \int_{0}^\infty \Big(\int_{\R^{n}} u\,(L^*\varphi+\de_{x_n}\varphi) \,\d x\Big)\d t\\
 &=\lim_{\epsilon\to 0^+}\int_{\epsilon}^\infty \Big(\int_{\R^{n}} u\,L^*\varphi\,\d x
 +\int_{\R^{n}} u\,\de_{t}\varphi \,\d x\Big) \d t\\
 &\quad \text{(by \eqref{formaLLLLCauchy}, $L$ operates only in the $x$-variable and integration by parts is allowed})\\
 &=\lim_{\epsilon\to 0^+}\int_{\epsilon}^\infty \Big(\int_{\R^{n}} L u\,\varphi\,\d x+\int_{\R^{n}} u\,\de_{t}
 \varphi \,\d x\Big) \d t\\
 &\quad \text{(we use \eqref{CAUUUeq} and \eqref{formaLLLLCauchy}, ensuring that $L u=\de_{t}u$ on $\Omega$)}\\
 &=\lim_{\epsilon\to 0^+}\int_{\epsilon}^\infty \Big(\int_{\R^{n}} \de_t(u\,\varphi)\,\d x\Big)\d t
 =\lim_{\epsilon\to 0^+} \int_{\R^{n}}\Big(\int_{\epsilon}^\infty \de_t(u\,\varphi)\,\d t\Big)\d x\\
 &=\lim_{\epsilon\to 0^+} -\int_{\R^{n}}u(\epsilon,x)\,\varphi(\epsilon,x)\,\d x=0.
\end{align*}
 In the last identity we used the initial condition of \eqref{CAUUUeq}
 and a simple dominated-convergence argument (since $u\in C(\overline{\Omega})$).
 This completes the proof.
\end{proof}
 We explicitly remark that, among our examples,
 the operators
 \begin{itemize}
   \item $\LL_2$ in Section \ref{sec:MEXPgroupsINV}-(a),
   \item $\LL_2$ and $\mathcal{K}$ in Section \ref{sec:MEXPgroupsINV}-(b),
   \item $\LL$ in Section \ref{sec:MEXPgroupsINV}-(c),
  \item $\mathcal{H}$ in Section \ref{sec:MEXPgroupsINV}-(d)
 \end{itemize}
 all satisfy the structure assumptions in
  \eqref{formaLLLLCauchy}, so that Proposition \ref{pop:cauchyU}
  can be applied to them. Corollary \ref{cor.tyko} in the Introduction
  is a particular case of
  Proposition \ref{pop:cauchyU} (obtained by replacing the matrix $B$
  with $-B$).
\bibliographystyle{alpha}

\end{document}